\documentclass{amsart}

\usepackage[english]{babel}
\usepackage[cp1250]{inputenc}
\usepackage{graphicx}
\usepackage{url}
\usepackage[all]{xy}
\usepackage{multicol}
\usepackage{amsthm}
\usepackage{amsmath}
\usepackage{amssymb}
\usepackage{geometry}
\usepackage{pinlabel}
\usepackage{hyperref}
\usepackage[pagewise]{lineno}

\newtheorem{Theorem}{Theorem}[section]
\newtheorem{lemma}[Theorem]{Lemma}
\newtheorem{proposition}[Theorem]{Proposition}
\newtheorem{corollary}[Theorem]{Corollary}
\newtheorem{definition}[Theorem]{Definition}
\newtheorem{remark}[Theorem]{Remark}
\newtheorem{example}[Theorem]{Example}

\newcommand{\RR}{\mathbb {R}}

\newcommand{\bi}{\widehat{\mathcal{B}}_{L}}

\makeatletter
\long\def\dow#1{\leavevmode\setbox\@tempboxa\hbox{$#1$}
\@tempdima\fboxrule
\advance\@tempdima \fboxsep \advance\@tempdima \dp\@tempboxa
\hbox{\hskip\fboxsep \lower \@tempdima\hbox
{\vbox{
\hbox{ 
\vbox{\vskip\fboxsep \box\@tempboxa\vskip\fboxsep}\hskip
\fboxsep\vrule \@width \fboxrule}%
\hrule \@height \fboxrule}}}\hskip\fboxsep }
\makeatother

\makeatletter
\long\def\dol#1{\leavevmode\setbox\@tempboxa\hbox{$#1$}
\@tempdima\fboxrule
\advance\@tempdima \fboxsep \advance\@tempdima \dp\@tempboxa
\hbox{\hskip\fboxsep \lower \@tempdima\hbox
{\vbox{
\hbox{\vrule \@width \fboxrule 
\hskip \fboxsep \vbox{\vskip\fboxsep \box\@tempboxa\vskip\fboxsep }\hskip\fboxsep }%
\hrule \@height \fboxrule}}} \hskip\fboxsep }
\makeatother

\makeatletter
\long\def\up#1{\leavevmode\setbox\@tempboxa\hbox{$#1$}
\@tempdima\fboxrule
\advance\@tempdima \fboxsep \advance\@tempdima \dp\@tempboxa
\hbox{\hskip\fboxsep \lower \@tempdima\hbox
{\vbox{\hrule \@height \fboxrule 
\hbox{ 
\vbox{\vskip\fboxsep \box\@tempboxa\vskip\fboxsep}\hskip
\fboxsep\vrule \@width \fboxrule}%
}}}\hskip\fboxsep }
\makeatother

\makeatletter
\long\def\upl#1{\leavevmode\setbox\@tempboxa\hbox{$#1$}
\@tempdima\fboxrule
\advance\@tempdima \fboxsep \advance\@tempdima \dp\@tempboxa
\hbox{\hskip\fboxsep \lower \@tempdima\hbox
{\vbox{\hrule \@height \fboxrule
\hbox{\vrule \@width \fboxrule 
\hskip \fboxsep \vbox{\vskip\fboxsep \box\@tempboxa\vskip\fboxsep }\hskip\fboxsep }%
}}}\hskip\fboxsep }
\makeatother

\begin{document}
\title{The topological biquandle of a link}

\author{Eva Horvat}
\address{University of Ljubljana\\
Faculty of Education\\
Kardeljeva plo\v s\v cad 16\\
1000 Ljubljana, Slovenia}
\email{eva.horvat@pef.uni-lj.si}

\date{\today}
\maketitle

\begin{abstract}
To every oriented link $L$, we associate a topologically defined biquandle $\bi $, which we call the topological biquandle of $L$. The construction of $\bi $ is similar to the topological description of the fundamental quandle given by Matveev. We find a presentation of the topological biquandle and explain how it is related to the fundamental biquandle of the link. 
\end{abstract}

\begin{section}{Introduction}
A biquandle is an algebraic structure with two operations that generalizes a quandle. The axioms of both structures represent an algebraic encoding of the Reidemeister moves, and study of quandles and related structures has been closely intertwined with knot theory. 

It is well known that every knot has a fundamental quandle, that admits an algebraic as well as a topological interpretation. Its topological description is due to Matveev \cite{MA}, who called it \textit{the geometric grupoid} of a knot and proved that the fundamental quandle is a complete knot invariant up to inversion (taking the mirror image and reversing orientation). 

The fundamental biquandle of a knot or link, however, is purely algebraically defined. It is not clear whether it also admits a topological interpretation \cite{KA}. Various other issues concerning biquandles have not yet been resolved, see \cite{UN}. 
 
To any classical oriented link, we associate a topologically defined biquandle $\bi $, which we call the topological biquandle of the link. Our construction is similar to Matveev's construction of the geometric grupoid of a knot. The topological construction enables us to visualize the biquandle operations directly and improves our understanding of the biquandle structure. Another advantage of this construction is that it defines a functor from the (topological) category of oriented links in $S^{3}$ to the category of biquandles. 

We show that the topological biquandle $\bi $ is a quotient of the fundamental biquandle, but its structure is simpler than that of a general biquandle.

This paper is organized as follows. In Section \ref{sec1}, we give the definition of a biquandle, recall some of its basic properties, define biquandle presentations and the fundamental biquandle of a link. Section \ref{sec2} is the core of the paper, in which we define the topological biquandle of a link, prove that it is a biquandle and study some of its properties. In Section \ref{sec3}, we investigate the topological biquandle from the perspective of a link diagram. We find a presentation of the topological biquandle and show that is is a quotient of the fundamental biquandle.  
\end{section}

\begin{section}{Preliminaries}
\label{sec1}
For an introduction to biquandles, we refer the reader to \cite{KAUF1}, \cite{KAUF2}, \cite{KAUF3}.
\begin{definition}[Biquandle axioms] \label{def1}
A \textbf{biquandle} is a set $B$ with two binary operations, the up operation $\up{\vbox to 0.2cm {\, }}$ and the down operation $\dow{\vbox to 0.2cm {\, }}$, such that $B$ is closed under these operations and that the following axioms are satisfied: 
\begin{enumerate}
\item For every $a\in B$, the maps $f_{a},g_{a}\colon B\to B$ and $S\colon B\times B\to B\times B$, defined by $f_{a}(x)=x\up{a}$, $g_{a}(x)=x\dow{a}$ and $S(x,y)=(y\dow{x},x\up{y})$, are bijections.
\item For every $a\in B$, we have $f_{a}^{-1}(a)=a\dow{f_{a}^{-1}(a)}$ and $g_{a}^{-1}(a)=a\up{g_{a}^{-1}(a)}$. 
\item For every $a,b,c\in B$, the equalities \begin{xalignat*}{2}
& \textrm{(Up Interchanges)} & a\up{b}\up{c}=a\up{c\dow{b}}\up{b\up{c}}\\
& \textrm{(Rule of Five)} & a\dow{b}\up{c\dow{b\up{a}}}=a\up{c}\dow{b\up{c\dow{a}}}\\
& \textrm{(Down Interchanges)} & a\dow{b}\dow{c}=a\dow{c\up{b}}\dow{b\dow{c}}
\end{xalignat*} are valid. 
\end{enumerate}
\end{definition}
A biquandle $(B,\up{\vbox to 0.2cm {\, }},\dow{\vbox to 0.2cm {\, }})$ in which $a\dow{b}=a$ for all $a\in B$ is called a \textbf{quandle}. \\

It follows from the first biquandle axiom that the map $S\colon B\times B\to B\times B$ has an inverse. Define two new operations $\upl{\vbox to 0.2cm {\, }}$ and $\dol{\vbox to 0.2cm {\, }}$ on $B$ by $$S^{-1}(a,b)=(b\upl{a},a\dol{b})\;.$$

\begin{remark} This ''corner'' notation was introduced by Kauffman \cite{KAUF1}. Another alternative is the ''exponential notation'' that was used by Fenn and Rourke in \cite{FR}, and avoids brackets. One may translate between the two notations using equalities: $a^{b}=a\up b$, $a^{\overline{b}}=a\upl b$, $a_{b}=a\dow b$ and $a_{\overline{b}}=a\dol b$. 
\end{remark}

\begin{lemma}\label{lemma0} For every $a,b\in B$, the equalities $$a\up{b}\upl{b\dow{a}}=a\upl{b}\up{b\dol{a}}=a\dow{b}\dol{b\up{a}}=a\dol{b}\dow{b\upl{a}}=a$$ are valid. 
\end{lemma}
\begin{proof}We compute
\begin{xalignat*}{1}
& (a,b)=S^{-1}\left (S(a,b)\right )=S^{-1}\left (b\dow{a},a\up{b}\right )=\left (a\up{b}\upl{b\dow{a}},b\dow{a}\dol{a\up{b}}\right )\\
& (a,b)=S\left (S^{-1}(a,b)\right )=S\left (b\upl{a},a\dol{b}\right )=\left (a\dol{b}\dow{b\upl{a}},b\upl{a}\up{a\dol{b}}\right )
\end{xalignat*} and the desired equalities follow. 
\end{proof}

\begin{lemma}\label{lemma01} Let $X$ and $Y$ be two biquandles. If $f\colon X\to Y$ is a biquandle homomorphism, then $f(a\upl{b})=f(a)\upl{f(b)}$ and $f(a\dol{b})=f(a)\dol{f(b)}$ for every $a,b\in X$.  
\end{lemma}
\begin{proof} Let $f\colon X\to Y$ be a biquandle homomorphism. Choose elements $a,b\in X$ and denote $f(b\upl{a})=x$ and $f(a\dol{b})=y$. By Lemma \ref{lemma0} we have $a\dol{b}\dow{b\upl{a}}=a$, and since $f$ is a biquandle homomorphism, it follows that $f(a\dol{b})\dow{f(b\upl{a})}=y\dow{x}=f(a)$. Also by Lemma \ref{lemma0}, we have $b\upl{a}\up{a\dol{b}}=b$, and since $f$ is a biquandle homomorphism, it follows that $f(b\upl{a})\up{f(a\dol{b})}=x\up{y}=f(b)$. Putting those two equalities together, Lemma \ref{lemma0} gives $$y=y\dow{x}\dol{x\up{y}}=f(a)\dol{f(b)} \textrm{  and  }x=x\up{y}\upl{y\dow{x}}=f(b)\upl{f(a)}\;.$$  
\end{proof} 

The fundamental biquandle of a link is usually defined via a presentation, coming from a link diagram. Following \cite{ISHI}, we define biquandle presentations categorically.
\begin{definition} Let $A$ be a set. A \textbf{free biquandle} on $A$ is the biquandle $F_{BQ}(A)$ together with an injective map $i\colon A\to F_{BQ}(A)$, characterized by the following. For any map $f\colon A\to B$, where $B$ is a biquandle, there exists a unique biquandle homomorphism $\overline{f}\colon F_{BQ}(A)\to B$ such that $f=\overline{f}\circ i$.  

For a biquandle $X$, let $j\colon A\to X$ be a map and let $\overline{j}\colon F_{BQ}(A)\to X$ be the induced biquandle homomorphism. Let $R\subset F_{BQ}(A)\times F_{BQ}(A)$ be a relation on the set $F_{BQ}(A)$. We say that $\langle A|R\rangle $ is a \textbf{presentation} of the biquandle $X$ if \begin{enumerate}
\item $(\overline{j}\times \overline{j})(R)\subset \Delta _{X}$ (here $\Delta _{X}\subset X\times X$ is the diagonal)
\item for any biquandle $Y$ and for any map $f\colon A\to Y$ such that $(\overline{f}\times \overline{f})(R)\subset \Delta _{Y}$, there exists a unique biquandle homomorphism $\widetilde{f}\colon X\to Y$ such that $f=\widetilde{f}\circ j$.  
\end{enumerate} 
\end{definition}
 
Any classical oriented link may be given by its diagram, ie. the image of a regular projection of the link to a plane in $\RR ^{3}$. A link diagram $D$ is a directed 4-valent graph, whose vertices contain the information about the over- and undercrossings. The edges of the graph are called \textbf{semiarcs}, while the vertices are called crossings of the diagram. Denote by $A(D)$ the set of semiarcs and by $C(D)$ the set of crossings of the diagram $D$. In any crossing, the four semiarcs are connected by two \textbf{crossing relations}, depicted in Figure \ref{fig:slika5}. 

\begin{definition} The \textbf{fundamental biquandle} $BQ(L)$ of a link $L$ with a diagram $D$ is the biquandle, given by the presentation $$\langle A(D)|\, \textrm{crossing relations for every }c\in C(D)\rangle \;.$$
\end{definition}

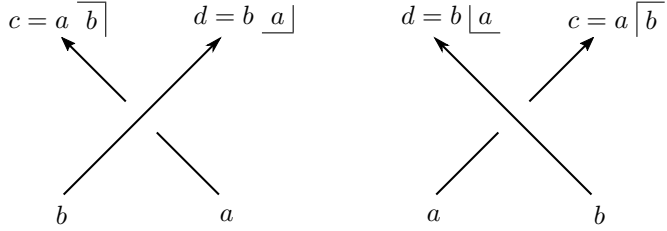
\begin{figure}[h!]
\labellist
\normalsize \hair 2pt
\pinlabel $b$ at 0 10
\pinlabel $a$ at 160 10
\pinlabel $c=a\up{b}$ at 0 200
\pinlabel $d=b\dow{a}$ at 180 200
\pinlabel $a$ at 360 10
\pinlabel $b$ at 520 10
\pinlabel $d=b\dol{a}$ at 380 200
\pinlabel $c=a\upl{b}$ at 540 200
\endlabellist
\begin{center}
\includegraphics[scale=0.4]{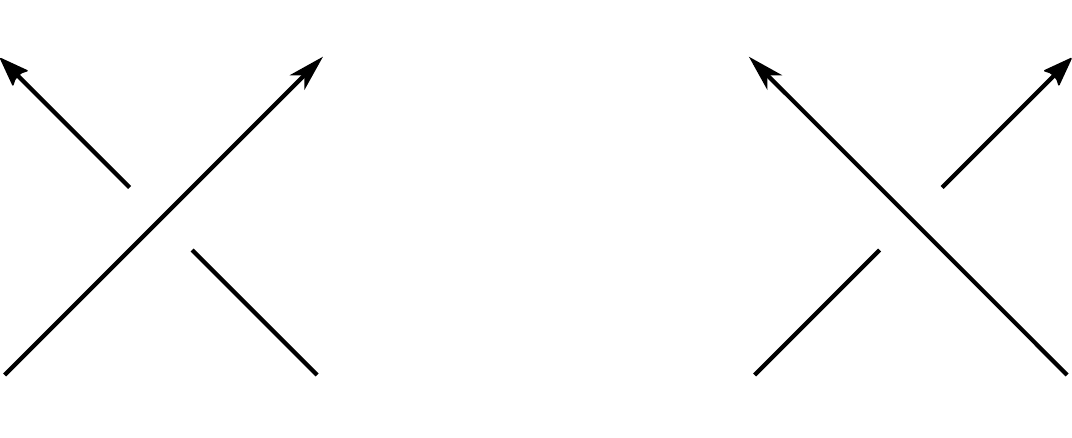}
\caption{Crossing relations between the semiarcs of $D$}
\label{fig:slika5}
\end{center}
\end{figure}
\end{section}

\begin{section}{The topological biquandle of a link}
\label{sec2}
By a link we will mean an oriented subspace of $S^{3}$, homeomorphic to a disjoint union of circles $\sqcup _{i=1}^{k}S^{1}$. For a link $L$, denote by $N_{L}$ a regular neighborhood of $L$ in $S^{3}$ and let $E_{L}=\textrm{closure}(S^{3}-N_{L})$. The orientation of $L$ induces an orientation of its normal bundle using the right-hand rule. 

Choose a 3-ball $B^{3}\subset S^{3}$ such that $N_{L}\subset B^{3}$, then let $z_{0}$ and $z_{1}$ be two antipodal points of $S^{2}=\partial B^{3}$. Define 
$$\mathcal{B}_{L}=\left \{(a_{0},a_{1})|\, a_{i}\colon [0,1]\to E_{L}\textrm{ a path from a point on $\partial N_{L}$ to $z_{i}$ for $i=0,1$ and $a_{0}(0)=a_{1}(0)$}\right \}\;.$$

If $a\colon [0,1]\to E_{L}$ is a path, we denote by $\overline{a}\colon [0,1]\to E_{L}$ the reverse path, given by $\overline{a}(t)=a(1-t)$. Given paths $a,b\colon [0,1]\to E_{L}$ with $a(1)=b(0)$, their combined path $a\cdot b$ is given by $$(a\cdot b)(t)=\left \{ \begin{array}{lr}
a(2t), & 0\leq t\leq \frac{1}{2};\\
b(2t-1), & \frac{1}{2}<t\leq 1.\\
\end{array}
\right .$$ 

We say that two elements $(a_{0},a_{1}),(b_{0},b_{1})\in \mathcal{B}_{L}$ are \textbf{equivalent} if there exists a homotopy $H_{t}\colon [0,1]\to E_{L}$ such that $H_{0}=\overline{a_{0}}\cdot a_{1}$, $H_{1}=\overline{b_{0}}\cdot b_{1}$, $H_{t}(0)=z_{0}$, $H_{t}(1)=z_{1}$ and $H_{t}(\frac{1}{2})\in \partial N_{L}$ for all $t\in [0,1]$. It is easy to see this defines an equivalence relation on the set $\mathcal{B}_{L}$. The quotient set $\widehat{\mathcal{B}}_{L}=\mathcal{B}_{L}/_{\sim }$ will be the underlying set of the topological biquandle of $L$. 

\begin{remark} Observe that every element of $\mathcal{B}_{L}$ is given by a pair of paths $(a_{0},a_{1})$ in $E_{L}$. The homotopy class of the path $a_{i}$ is an element of the fundamental quandle $Q(L)$ with the basepoint $z_{i}$ for $i=0,1$. We thus obtained the set $\bi $ by taking pairs of representatives of the fundamental quandle $Q(L)$, and then imposing on those pairs a new equivalence relation.  
\end{remark} 

The set $\bi $ is closely related to the group of the link $L$. For any point $p\in \partial N_{L}$, denote by $m_{p}$ the loop in $\partial N_{L}$, based at $p$, which goes once around the meridian of $L$ in the positive direction according to the orientation of the normal bundle. Define two maps $p_{i}\colon \mathcal{B}_{L}\to \pi _{1}(E_{L},z_{i})$ by $p_{i}(a_{0},a_{1})=[\overline{a_{i}}\cdot m_{a_{i}(0)}\cdot a_{i}]$ for $i=0,1$.  

\begin{lemma}\label{lemma1} If $(a_{0},a_{1})\sim (b_{0},b_{1})$, then $p_{i}(a_{0},a_{1})=p_{i}(b_{0},b_{1})$ for $i=0,1$. 
\end{lemma}
\begin{proof} Let $(a_{0},a_{1})\sim (b_{0},b_{1})$ be two equivalent elements of $\mathcal{B}_{L}$. Then there exists a homotopy $H_{t}\colon [0,1]\to E_{L}$ such that $H_{0}=\overline{a_{0}}\cdot a_{1}$, $H_{1}=\overline{b_{0}}\cdot b_{1}$, $H_{t}(0)=z_{0}$, $H_{t}(1)=z_{1}$ and $H_{t}(\frac{1}{2})\in \partial N_{L}$ for all $t\in [0,1]$. It follows that $a_{0}(0)$ and $b_{0}(0)$ lie in the same boundary component of $\partial N_{L}$. Since $m_{a_{0}(0)}$ and $m_{b_{0}(0)}$ are two meridians of the same component of $L$, we may choose a homotopy $G_{t}\colon [0,1]\to \partial N_{L}$ such that $G_{0}=m_{a_{0}(0)}$, $G_{1}=m_{b_{0}(0)}$ and $G_{t}(0)=G_{t}(1)=H_{t}(\frac{1}{2})$ for $t\in [0,1]$. Similarly, we may choose a homotopy $J_{t}\colon [0,1]\to \partial N_{L}$ such that $J_{0}=m_{a_{1}(0)}$, $J_{1}=m_{b_{1}(0)}$ and $J_{t}(0)=J_{t}(1)=H_{t}(\frac{1}{2})$ for $t\in [0,1]$. Define a map $S_{t}\colon [0,1]\to E_{L}$ by 
$$S_{t}(u)=\left \{ \begin{array}{lr}
H_{t}\left (\frac{3u}{2}\right ), & 0\leq u\leq \frac{1}{3};\\
G_{t}(3u-1), & \frac{1}{3}\leq u\leq \frac{2}{3};\\
H_{t}\left (\frac{3(1-u)}{2}\right ), &  \frac{2}{3}\leq u\leq 1.
\end{array}
\right .$$ 
Now $S_{t}$ is a homotopy between the loops $\overline{a_{0}}\cdot m_{a_{0}(0)}\cdot a_{0}$ and $\overline{b_{0}}\cdot m_{b_{0}(0)}\cdot b_{0}$, which thus represent the same element of the fundamental group $\pi _{1}(E_{L},z_{0})$. It follows that $p_{0}(a_{0},a_{1})=p_{0}(b_{0},b_{1})$. The proof for $i=1$ is similar. 

\end{proof}

\begin{corollary} The map $p_{i}$ induces a map $\widehat{p}_{i}\colon \bi \to \pi _{1}(E_{L},z_{i})$ for $i=0,1$. 
\end{corollary}

Denote by $[a_{0},a_{1}]\in \bi $ the equivalence class of the element $(a_{0},a_{1})\in \mathcal{B}_{L}$. We have found a way to associate to each element $[a_{0},a_{1}]$ of the set $\widehat{\mathcal{B}}_{L}$ two elements of the fundamental groups $\pi _{1}(E_{L},z_{0})$ and $\pi _{1}(E_{L},z_{1})$, namely $\widehat{p}_{0}[a_{0},a_{1}]$ and $\widehat{p}_{1}[a_{0},a_{1}]$. Using this association, we will now define the operations on $\bi $. 

Define two binary operations (called the up- and down- operation) on $\mathcal{B}_{L}$ by 
\begin{xalignat*}{3}
(a_{0},a_{1})^{(b_{0},b_{1})}:=(a_{0}\cdot p_{0}(b_{0},b_{1}),a_{1}) & \quad \textrm{ and}  & (a_{0},a_{1})_{(b_{0},b_{1})}:=(a_{0},a_{1}\cdot p_{1}(b_{0},b_{1}))\;.
\end{xalignat*} 
We intend to show that these operations induce operations on the quotient space $\bi $, and that $\bi $ equipped with those operations forms a biquandle. 

\begin{lemma}\label{lemma2} If $(a_{0},a_{1})\sim (c_{0},c_{1})$ and $(b_{0},b_{1})\sim (d_{0},d_{1})$, then $(a_{0},a_{1})^{(b_{0},b_{1})}\sim (c_{0},c_{1})^{(d_{0},d_{1})}$ and $(a_{0},a_{1})_{(b_{0},b_{1})}\sim (c_{0},c_{1})_{(d_{0},d_{1})}$. 
\end{lemma}
\begin{proof} Let $(a_{0},a_{1})\sim (c_{0},c_{1})$ and $(b_{0},b_{1})\sim (d_{0},d_{1})$ in $\mathcal{B}_{L}$. There is a homotopy $H_{t}\colon [0,1]\to E_{L}$ such that $H_{0}=\overline{a_{0}}\cdot a_{1}$, $H_{1}=\overline{c_{0}}\cdot c_{1}$, $H_{t}(0)=z_{0}$, $H_{t}(1)=z_{1}$ and $H_{t}(\frac{1}{2})\in \partial N_{L}$ for all $t\in [0,1]$. Since $(b_{0},b_{1})\sim (d_{0},d_{1})$, it follows by Lemma \ref{lemma1} that there exists a homotopy $G_{t}\colon [0,1]\to E_{L}$ such that $G_{0}=p_{0}(b_{0},b_{1})$, $G_{1}=p_{0}(d_{0},d_{1})$ and $G_{t}(0)=G_{t}(1)=z_{0}$ for all $t\in [0,1]$. Define a map $S_{t}\colon [0,1]\to E_{L}$ by 
$$S_{t}(u)=\left \{ \begin{array}{lr}
G_{t}(1-4u), & 0\leq u\leq \frac{1}{4};\\
H_{t}(2u-\frac{1}{2}), & \frac{1}{4}\leq u\leq \frac{1}{2};\\
H_{t}(u), &  \frac{1}{2}\leq u\leq 1.
\end{array}
\right .$$ Now $S_{t}$ is a homotopy from $\overline{a_{0}\cdot p_{0}(b_{0},b_{1})}\cdot a_{1}$ to $\overline{c_{0}\cdot p_{0}(d_{0},d_{1})}\cdot c_{1}$, for which $S_{t}(0)=z_{0}$, $S_{t}(1)=z_{1}$ and $S_{t}\left (\frac{1}{2}\right )\in \partial N_{L}$ for all $t\in [0,1]$. It follows that $(a_{0},a_{1})^{(b_{0},b_{1})}\sim (c_{0},c_{1})^{(d_{0},d_{1})}$. The proof is similar for $(a_{0},a_{1})_{(b_{0},b_{1})}\sim (c_{0},c_{1})_{(d_{0},d_{1})}$.
\end{proof}

\begin{corollary} There are induced up- and down- operations on $\bi $, defined by 
\begin{xalignat*}{3}
[a_{0},a_{1}]^{[b_{0},b_{1}]}:=[a_{0}\cdot p_{0}(b_{0},b_{1}),a_{1}]\; & \quad \textrm{ and} & [a_{0},a_{1}]_{[b_{0},b_{1}]}:=[a_{0},a_{1}\cdot p_{1}(b_{0},b_{1})]\;.
\end{xalignat*}
\end{corollary}

\begin{lemma}\label{lemma3} The maps $f_{a},g_{a}\colon \bi \to \bi $, defined by $f_{a}(x)=x^{a}$ and $g_{a}(x)=x_{a}$, are bijective for any $a\in \bi $. 
\end{lemma}
\begin{proof}Define maps $f_{a}',g_{a}'\colon \bi \to \bi $ by $f_{a}'([b_{0},b_{1}])=[a_{0}\cdot \overline{p_{0}(b_{0},b_{1})},a_{1}]$ and $g_{a}'([b_{0},b_{1}])=[a_{0},a_{1}\cdot \overline{p_{1}(b_{0},b_{1})}]$. It is easy to see that $f_{a}'$ is the inverse of $f_{a}$ and $g_{a}'$ is the inverse of $g_{a}$, thus $f_{a}$ and $g_{a}$ are bijective. 
\end{proof}

\begin{Theorem} \label{th1} The set $\bi $, equipped with the induced up- and down- operations, is a biquandle. 
\end{Theorem}
\begin{proof} For any $a,b\in \bi $, denote $a\up{b}:=a^{b}$ and $a\dow{b}:=a_{b}$. We need to show that $\bi $ equipped with those operations satisfies all the biquandle axioms. \\
(1)  Let $a\in \bi $. The maps $f_{a},g_{a}\colon \bi \to \bi $, defined by $f_{a}(x)=x\up{a}$ and $g_{a}(x)=x\dow{a}$, are bijective by Lemma \ref{lemma3}. The map $S\colon \bi \times \bi \to \bi \times \bi $ is defined by $S(a,b)=(b\dow{a},a\up{b})$. Consider another map $T\colon \bi \times \bi \to \bi \times \bi $, defined by $T\left ([a_{0},a_{1}],[b_{0},b_{1}]\right )=\left ([b_{0}\cdot \overline{p_{0}(a_{0},a_{1})},b_{1}],[a_{0},a_{1}\cdot \overline{p_{1}(b_{0},b_{1}}]\right )$, and compute
\begin{xalignat*}{1}
& T\left (S([a_{0},a_{1}],[b_{0},b_{1}])\right )=T([b_{0},b_{1}\cdot p_{1}(a_{0},a_{1})],[a_{0}\cdot p_{0}(b_{0},b_{1}),a_{1}])=\\
& =([a_{0}\cdot p_{0}(b_{0},b_{1})\cdot \overline{p_{0}(b_{0},b_{1}\cdot p_{1}(a_{0},a_{1}))},a_{1}],[b_{0},b_{1}\cdot p_{1}(a_{0},a_{1})\cdot \overline{p_{1}(a_{0}\cdot p_{0}(b_{0},b_{1}),a_{1})})=\\
& =([a_{0}\overline{b}_{0}m_{b_{0}(0)}b_{0}\overline{b}_{0}\overline{m}_{b_{0}(0)}b_{0},a_{1}],[b_{0},b_{1}\overline{a}_{1}m_{a_{1}(0)}a_{1}\overline{a}_{1}\overline{m}_{a_{1}(0)}a_{1}])=\left ([a_{0},a_{1}],[b_{0},b_{1}]\right )
\end{xalignat*} A similar calculation shows that $ST=id$, thus $S$ is bijective with inverse $T$. \\
(2) Let $a=[a_{0},a_{1}]\in \bi $. We calculate \begin{xalignat*}{1}
& f_{a}^{-1}(a)=[a_{0},a_{1}]^{\overline{[a_{0},a_{1}]}}=[a_{0}\overline{a}_{0}\overline{m}_{a_{0}(0)}a_{0},a_{1}]=[\overline{m}_{a_{0}(0)}a_{0},a_{1}]\\
& a\dow{f_{a}^{-1}(a)}=[a_{0},a_{1}]_{\left ([a_{0},a_{1}]^{\overline{[a_{0},a_{1}]}}\right )}=[a_{0},a_{1}]_{[a_{0}\overline{a}_{0}\overline{m}_{a_{0}(0)}a_{0},a_{1}]}=[a_{0},a_{1}\overline{a}_{1}m_{a_{1}(0)}a_{1}]=[a_{0},m_{a_{1}(0)}a_{1}]
\end{xalignat*}
Since $a_{0}(0)=a_{1}(0)$, we have $m_{a_{0}(0)}=m_{a_{1}(0)}$ and therefore the path $\overline{\overline{m}_{a_{0}(0)}a_{0}}a_{1}$ is homotopic to the path $\overline{a}_{0}m_{a_{1}(0)}a_{1}$. It follows that $f_{a}^{-1}(a)=a\dow{f_{a}^{-1}(a)}$. The proof of $g_{a}^{-1}(a)=a\up{g_{a}^{-1}(a)}$ is similar. \\
(3) Let $a=[a_{0},a_{1}]$, $b=[b_{0},b_{1}]$ and $c=[c_{0},c_{1}]$ be elements of $\bi $. Then we have \begin{xalignat*}{1}
& a\up{c\dow{b}}\up{b\up{c}}=\left (a^{c_{b}}\right )^{(b^{c})}=\left ([a_{0},a_{1}]^{[c_{0},c_{1}\overline{b}_{1}m_{b_{1}(0)}b_{1}]}\right )^{[b_{0}\overline{c}_{0}m_{c_{0}(0)}c_{0},b_{1}]}=\\
& =[a_{0}\overline{c}_{0}m_{c_{0}(0)}c_{0},a_{1}]^{[b_{0}\overline{c}_{0}m_{c_{0}(0)}c_{0},b_{1}]}=[a_{0}\overline{c}_{0}m_{c_{0}(0)}c_{0}\overline{c}_{0}\overline{m}_{c_{0}(0)}c_{0}\overline{b}_{0}m_{b_{0}(0)}b_{0}\overline{c}_{0}m_{c_{0}(0)}c_{0},a_{1}]=\\
& =[a_{0}\overline{b}_{0}m_{b_{0}(0)}b_{0}\overline{c}_{0}m_{c_{0}(0)}c_{0},a_{1}]=\left ([a_{0},a_{1}]^{[b_{0},b_{1}]}\right )^{[c_{0},c_{1}]}=a\up{b}\up{c}\\
& a\dow{b}\up{c\dow{b\up{a}}}=\left (a_{b}\right )^{c_{(b^{a})}}=[a_{0},a_{1}\overline{b}_{1}m_{b_{1}(0)}b_{1}]^{[c_{0},c_{1}]_{[b_{0}\overline{a}_{0}m_{a_{0}(0)}a_{0},b_{1}]}}=\\
& =[a_{0},a_{1}\overline{b}_{1}m_{b_{1}(0)}b_{1}]^{[c_{0},c_{1}\overline{b}_{1}m_{b_{1}(0)}b_{1}]}=[a_{0}\overline{c}_{0}m_{c_{0}(0)}c_{0},a_{1}\overline{b}_{1}m_{b_{1}(0)}b_{1}]=\\
& [a_{0}\overline{c}_{0}m_{c_{0}(0)}c_{0},a_{1}]_{[b_{0}\overline{c}_{0}m_{c_{0}(0)}c_{0},b_{1}]}= [a_{0}\overline{c}_{0}m_{c_{0}(0)}c_{0},a_{1}]_{\left ([b_{0},b_{1}]^{[c_{0},c_{1}\overline{a}_{1}m_{a_{1}(0)}a_{1}]}\right )}=\\
& =\left ([a_{0},a_{1}]^{[c_{0},c_{1}]}\right )_{\left ([b_{0},b_{1}]^{[c_{0},c_{1}]_{[a_{0},a_{1}]}}\right )}=a\up{c}\dow{b\up{c\dow{a}}}\\
\end{xalignat*} A similar calculation proves the Down Interchanges equality $ a\dow{c\up{b}}\dow{b\dow{c}}=a\dow{b}\dow{c}$. Therefore $\bi $ is a biquandle.
\end{proof}

Since $\bi $ is a biquandle, there are two more operations $\upl{\vbox to 0.2cm {\, }}$ and $\dol{\vbox to 0.2cm {\, }}$ on $\bi $, defined by $S^{-1}(a,b)=(b\upl{a},a\dol{b})$. We call those operations the up-bar and the down-bar operation respectively. It follows from the proof of Theorem \ref{th1} that the bar operations are computed as \begin{xalignat*}{1}
& [a_{0},a_{1}]\upl{[b_{0},b_{1}]}=[a_{0},a_{1}]^{\overline{[b_{0},b_{1}]}}=[a_{0}\cdot \overline{p_{0}(b_{0},b_{1})},a_{1}] \textrm{ and }\quad  \\
& [a_{0},a_{1}]\dol{[b_{0},b_{1}]}=[a_{0},a_{1}]_{\overline{[b_{0},b_{1}]}}=[a_{0},a_{1}\cdot \overline{p_{1}(b_{0},b_{1})}]\;.
\end{xalignat*}

\begin{definition} Biquandle $\bi $ is called the \textbf{topological biquandle} of the link $L$. 
\end{definition}

Observe that in the case of the topological biquandle, the name \textit{biquandle} becomes further justified, since every element of $\bi $ is represented by an ordered pair of paths (whose homotopy classes represent the elements of the fundamental quandle). We might ask ourselves which biquandles could be constructed from two quandles in a similar way. In \cite{EH} it is shown that given two quandles $Q$ and $K$, one may construct a product biquandle with underlying set $Q\times K$, whose operations are induced by the operations on $Q$ and $K$. Product biquandles are classified in \cite[Theorem 5.3]{EH}. 

In the remainder of this Section, we study properties of the topological biquandle $\bi $. It turns out that its structure is quite simpler than that of a general biquandle. 

\begin{lemma} In the topological biquandle, for any $a,b,c\in \bi $ the following holds:
\begin{enumerate}
\setlength{\itemsep}{2pt}
\item Any up- operation commutes with any down- operation,
\item $a\up{b}\upl{b}=a\upl{b}\up{b}=a\dow{b}\dol{b}=a\dol{b}\dow{b}=a$,
\item \begin{xalignat*}{3}
& a\up{b\dow{c}}=a\up{b\dol{c}}=a\up{b} & \quad & a\upl{b\dow{c}}=a\upl{b\dol{c}}=a\upl{b} \\
& a\dow{b\up{c}}=a\dow{b\upl{c}}=a\dow{b} & \quad & a\dol{b\up{c}}=a\dol{b\upl{c}}=a\dol{b}
\end{xalignat*} 
\end{enumerate} 
\label{lemma4}
\end{lemma}
\begin{proof} (1) For any $[a_{0},a_{1}],[b_{0},b_{1}],[c_{0},c_{1}]\in \bi $ we have
\begin{xalignat*}{1}
& \left ([a_{0},a_{1}]^{[b_{0},b_{1}]}\right )_{[c_{0},c_{1}]}=[a_{0}\cdot p_{0}(b_{0},b_{1}),a_{1}\cdot p_{1}(c_{0},c_{1})]=\left ([a_{0},a_{1}]_{[c_{0},c_{1}]}\right )^{[b_{0},b_{1}]}\;,
\end{xalignat*} and similar equalities hold for the up- bar and down- bar operations.  \\
(2) Let $a=[a_{0},a_{1}],b=[b_{0},b_{1}]\in \bi $ and compute 
\begin{xalignat*}{1}
& a\up{b}\upl{b}=\left ([a_{0},a_{1}]^{[b_{0},b_{1}]}\right )^{\overline{[b_{0},b_{1}]}}=[a_{0}\cdot \overline{b}_{0}m_{b_{0}(0)}b_{0}\cdot \overline{b}_{0}\overline{m}_{b_{0}(0)}b_{0},a_{1}]=[a_{0},a_{1}]=a\;,
\end{xalignat*} and similarly in the other three cases.  \\
(3)  We have $$[a_{0},a_{1}]^{\left ([b_{0},b_{1}]_{[c_{0},c_{1}]}\right )}=[a_{0},a_{1}]^{[b_{0},b_{1}\cdot p_{1}(c_{0},c_{1})]}=[a_{0}\cdot \overline{b}_{0}m_{b_{0}(0)}b_{0},a_{1}]=[a_{0},a_{1}]^{[b_{0},b_{1}]}\;,$$ and similar calculations settle the other  cases.
\end{proof}

\begin{proposition} \label{prop1} Let $(X,\up{\vbox to 0.2cm {\, }},\dow{\vbox to 0.2cm {\, }})$ be any biquandle in which the equalities $a\up{b\dow{c}}=a\up{b}$, $a\upl{b\dow{c}}=a\upl{b}$, $a\dow{b\up{c}}=a\dow{b}$ and $a\dol{b\up{c}}=a\dol{b}$ are valid for any $a,b,c\in X$. Then \begin{enumerate}
\item the equalities (3) from Lemma \ref{lemma4} are valid for any $a,b,c\in X$, 
\item for any $a,b\in X$ we have $a\up{b}\upl{b}=a\upl{b}\up{b}=a\dow{b}\dol{b}=a\dol{b}\dow{b}=a$,
\item any up- operation on $X$ commutes with any down- operation,
\item for any $a,b,c\in X$ we have $a\up{b\up{c}}=a\upl{c}\up{b}\up{c}$ and $a\dow{b\dow{c}}=a\dol{c}\dow{b}\dow{c}$. 
\end{enumerate}
\end{proposition}
\begin{proof} Let $X$ be a biquandle with the prescribed property. To prove (1), we use Lemma \ref{lemma0} to compute $a\up{b\dol{c}}=a\up{b\dol{c}\dow{c\upl{b}}}=a\up{b}$ and similarly for the other three cases. \\
To prove (2), choose elements $a,b\in X$ and use Lemma \ref{lemma0} to compute 
\begin{xalignat*}{2}
& a\up{b}\upl{b}=a\up{b}\upl{b\dow{a}}=a\;, &  a\upl{b}\up{b}=a\upl{b}\up{b\dol{a}}=a\;,\\
& a\dow{b}\dol{b}=a\dow{b}\dol{b\up{a}}=a\;, &  a\dol{b}\dow{b}=a\dol{b}\dow{b\upl{a}}=a\;.
\end{xalignat*}
To prove (3), choose elements $a,b,c\in X$ and use the second equality of the 3.Biquandle axiom to compute 
\begin{xalignat}{1} \label{eq1}
& a\up{c}\dow{b}=a\up{c}\dow{b\up{c\dow{a}}}=a\dow{b}\up{c\dow{b\up{a}}}=a\dow{b}\up{c}\;.
\end{xalignat} Now write $x=a\upl{b}\dol{c}$ and use (2) together with \eqref{eq1} to obtain $a=x\dow{c}\up{b}=x\up{b}\dow{c}$, which implies $x=a\dol{c}\upl{b}=a\upl{b}\dol{c}$.

Writing $y=a\upl{c}\dow{b}$, we use (2) and the second equality of the 3.Biquandle axiom to compute $$y\up{c}=y\up{c\dow{b\up{a}}}=a\upl{c}\dow{b}\up{c\dow{b\up{a}}}=a\upl{c}\up{c}\dow{b\up{c\dow{a}}}=a\dow{b}\;,$$ which implies $y=a\dow{b}\upl{c}=a\upl{c}\dow{b}$.  

Finally, write $z=a\up{b}\dol{c}$ and use the previously proved equality to obtain $a=z\dow{c}\upl{b}=z\upl{b}\dow{c}$, which implies $z=a\dol{c}\up{b}=a\up{b}\dol{c}$. 

To prove (4), choose elements $x,y,z\in X$ and use the first equality of the 3.Biquandle axiom to compute $x\up{y}\up{z\up{y}}=x\up{y\dow{z}}\up{z\up{y}}=x\up{z}\up{y}$ and putting $a=x\up{y}$, $b=z$ and $c=y$ gives $a\up{b\up{c}}=a\upl{c}\up{b}\up{c}$. Similarly, the third equality of the 3.Biquandle axiom gives $x\dow{y}\dow{z\dow{y}}=x\dow{y\up{z}}\dow{z\dow{y}}=x\dow{z}\dow{y}$ and putting $a=x\dow{y}$, $b=z$ and $c=y$ implies $a\dow{b\dow{c}}=a\dol{c}\dow{b}\dow{c}$.
\end{proof}

Part (3) of Lemma \ref{lemma4} together with Proposition \ref{prop1} implies the following:

\begin{corollary} \label{cor1} Let $A$ be a generating set of the topological biquandle $\bi $. Any element of $\bi $ can be expressed in the form $a\up{w_{1}}\dow{w_{2}}$, where $a\in A$ and $w_{i}$ is a word in $F(A)$ for $i=1,2$.  
\end{corollary}

\end{section}

\begin{section}{Presentation of the topological biquandle}
\label{sec3}
Recall the setting described at the beginning of Section \ref{sec2}. For a link $L$ in $S^{3}$, we have chosen a regular neighborhood $N_{L}$ and fixed an orientation of the normal bundle of $L$. We have also chosen the basepoints $z_{0}$ and $z_{1}$, which represent two antipodal points of the boundary sphere of a 3-ball neighborhood of $N_{L}$. Choose a coordinate system in which the points $z_{0}$ and $z_{1}$ have coordinates $(0,0,1)$ and $(0,0,-1)$ respectively, and let $D$ be the diagram of $L$ obtained by projection to the plane $z=0$.  

As before, we denote by $A(D)$ the set of semiarcs and by $C(D)$ the set of crossings of the diagram $D$. We would like to find a presentation of the topological biquandle $\bi $ in terms of the link diagram.

For any $a,b,c\in A(D)$, denote by $R_{a,b,c}$ the set of relations 
\begin{xalignat*}{1}
R_{a,b,c}=\left \{ a\up{b\dow{c}}=a\up{b},\, a\upl{b\dow{c}}=a\upl{b},\, a\dow{b\up{c}}=a\dow{b},\, a\dol{b\up{c}}=a\dol{b}\right \}
\end{xalignat*}

\begin{Theorem} Let $D$ be a diagram of a link $L$ in $S^{3}$. Then $$\left \langle A(D)|\, \textrm{crossing relations for each $c\in C(D)\,, R_{a,b,c}$ for each $a,b,c\in A(D)$}\right \rangle $$ is a presentation of the topological biquandle $\bi $. 
\end{Theorem}
\begin{proof} Let $R=\{\textrm{crossing relations for each }c\in C(D)\,, R_{a,b,c}$ for each $a,b,c\in A(D)\}$. We will define a map $j\colon A(D)\to \bi $ such that \begin{enumerate}
\item $(\overline{j}\times \overline{j})(R)\subset \Delta _{\bi }$,
\item for any biquandle $Y$ and for any map $f\colon A(D)\to Y$ such that $(\overline{f}\times \overline{f})(R)\subset \Delta _{Y}$, there exists a unique biquandle homomorphism $\widetilde{f}\colon \bi \to Y$ such that $f=\widetilde{f}\circ j$.  
\end{enumerate} 
For a semiarc $a\in A(D)$, let $j(a)=[a_{0},a_{1}]$, where $a_{0}$ is any path from the parallel curve to the semiarc $a$ to $z_{0}$ that passes \underline{over} all the other arcs of the diagram, and $a_{1}$ is a path from $a_{0}(0)$ to $z_{1}$ that passes \underline{under} all the other arcs of the diagram. 
\begin{enumerate}
\item[(proof of 1.)]
By definition of a free biquandle, there exists a unique biquandle homomorphism \\$\overline{j}\colon F_{BQ}(A(D))\to \bi $ that extends the map $j$, and it is given by \begin{xalignat*}{2}
& \overline{j}(a\up{b})=j(a)^{j(b)}=[a_{0},a_{1}]^{[b_{0},b_{1}]}\,, &  \overline{j}(a\dow{b})=j(a)_{j(b)}=[a_{0},a_{1}]_{[b_{0},b_{1}]}\;.
\end{xalignat*} It follows from Lemma \ref{lemma01} that $\overline{j}$ also satisfies 
\begin{xalignat*}{2}
& \overline{j}(a\upl{b})=j(a)^{\overline{j(b)}}=[a_{0},a_{1}]^{\overline{[b_{0},b_{1}]}}\;, &  \overline{j}(a\dol{b})=j(a)_{\overline{j(b)}}=[a_{0},a_{1}]_{\overline{[b_{0},b_{1}]}}\;.
\end{xalignat*} For any $a,b,c\in A(D)$, we use part (3) of Lemma \ref{lemma4} to compute $$\overline{j}(a\up{b\dow{c}})=j(a)\up{j(b)\dow{j(c)}}=j(a)\up{j(b)}=\overline{j}(a\up{b})\;,$$ and a similar computation shows that the homomorphism $\overline{j}$ preserves every relation from the set $R_{a,b,c}$. 

At every positive crossing of the diagram $D$, the outgoing semiarcs $c$ and $d$ are related to the incoming semiarcs $a$ and $b$ by two crossing relations $c=a\up{b}$ and $d=b\dow{a}$ (see the left part of Figure \ref{fig:slika5}). Figure \ref{fig:slika3} shows a homotopy between $\overline{j}(a\up{b})$ and $\overline{j}(c)$ and another homotopy between $\overline{j}(b\dow{a})$ and $\overline{j}(d)$. 
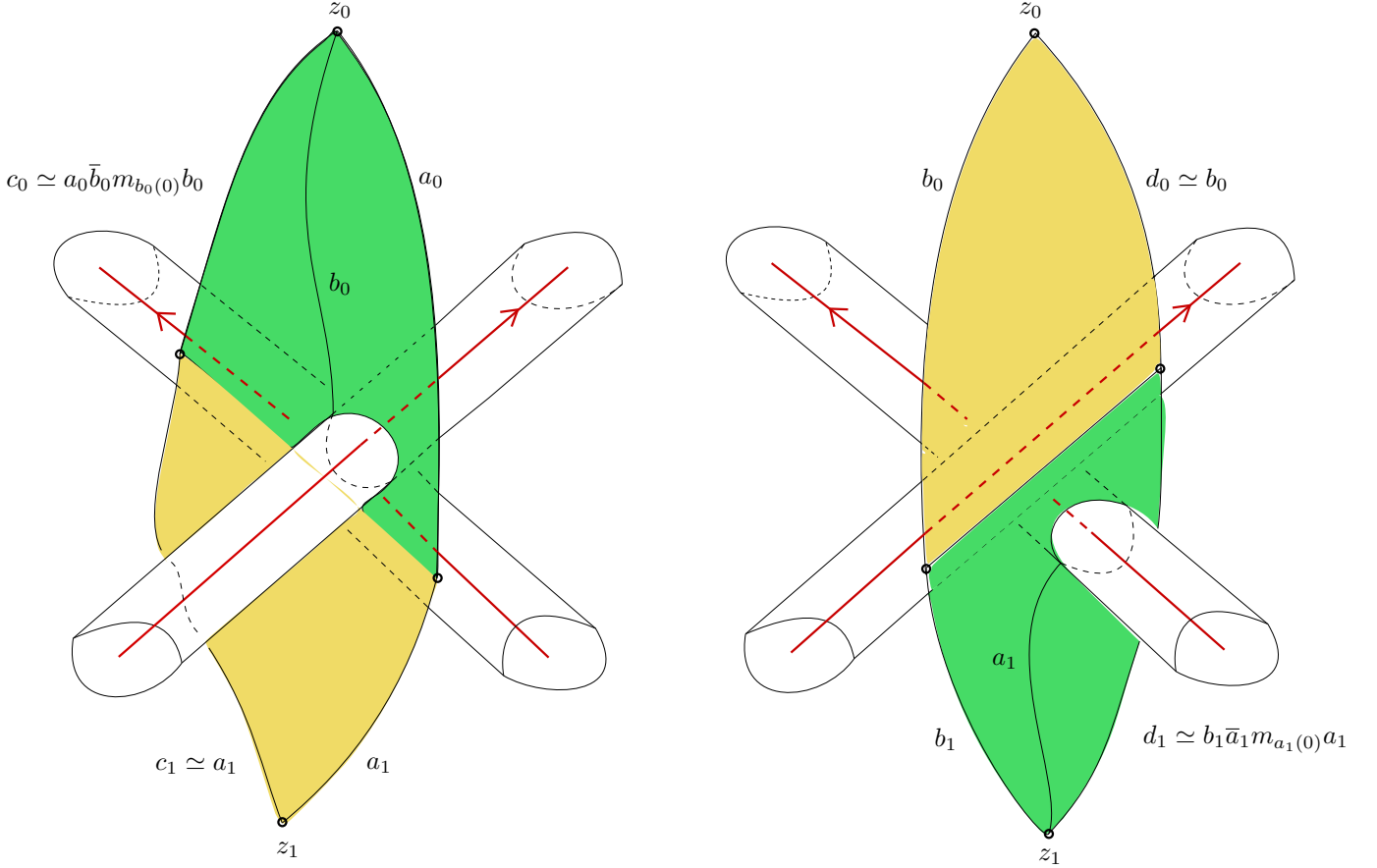
\begin{figure}[h!]
\labellist
\normalsize \hair 2pt
\pinlabel $z_{0}$ at 220 650
\pinlabel $z_{1}$ at 180 5
\pinlabel $c_{0}\simeq a_{0}\overline{b}_{0}m_{b_{0}(0)}b_{0}$ at 40 520 
\pinlabel $c_{1}\simeq a_{1}$ at 110 70
\pinlabel $b_{0}$ at 220 440 
\pinlabel $a_{0}$ at 290 520
\pinlabel $a_{1}$ at 250 70 
\pinlabel $z_{0}$ at 750 650
\pinlabel $z_{1}$ at 765 0
\pinlabel $b_{0}$ at 675 520 
\pinlabel $b_{1}$ at 685 90
\pinlabel $a_{1}$ at 730 150 
\pinlabel $d_{0}\simeq b_{0}$ at 870 520
\pinlabel $d_{1}\simeq b_{1}\overline{a}_{1}m_{a_{1}(0)}a_{1}$ at 915 90 
\endlabellist
\begin{center}
\includegraphics[scale=0.5]{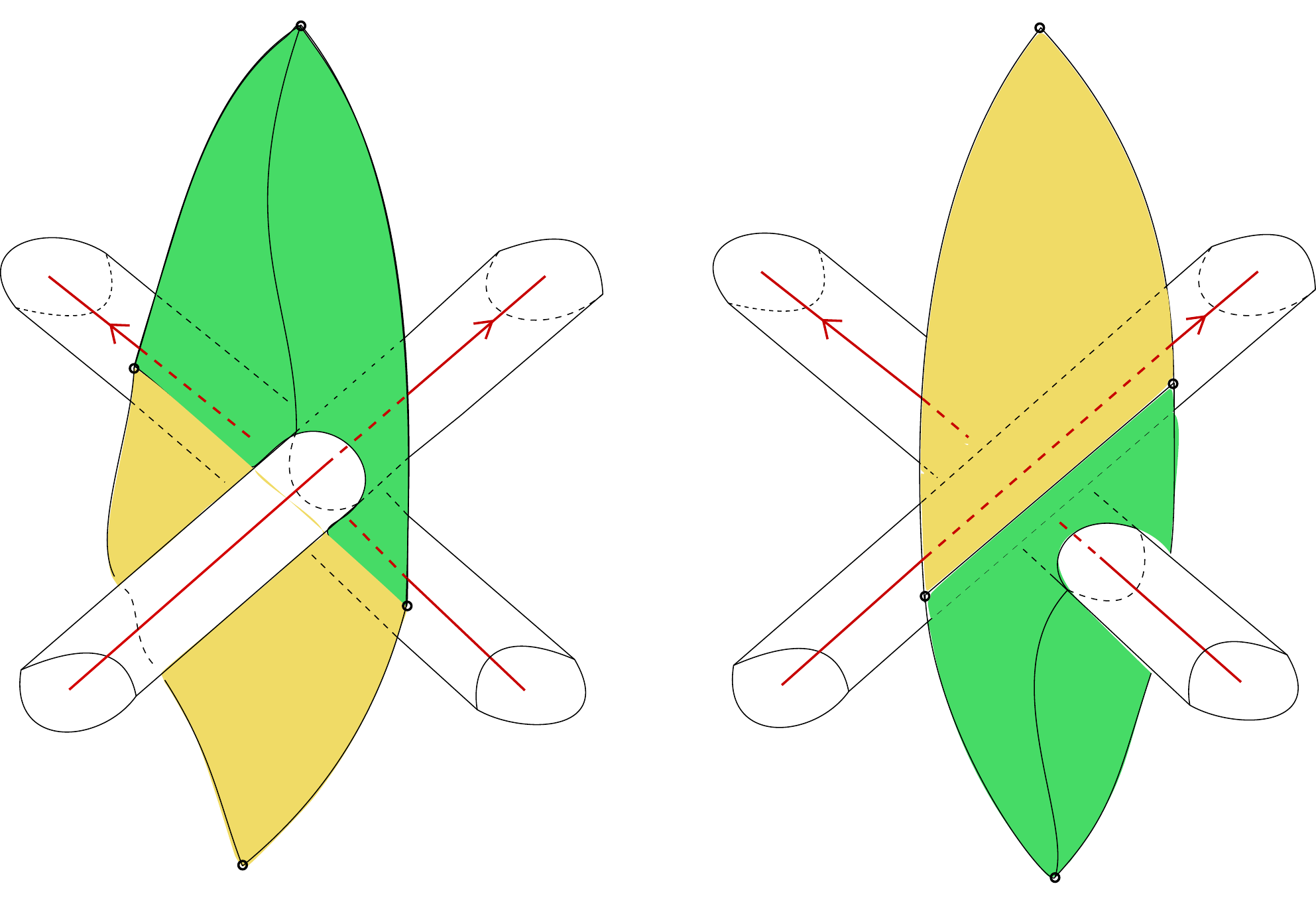}
\caption{An illustration of the crossing relations $[a_{0},a_{1}]^{[b_{0},b_{1}]}=[c_{0},c_{1}]$ and $[b_{0},b_{1}]_{[a_{0},a_{1}]}=[d_{0},d_{1}]$}
\label{fig:slika3}
\end{center}
\end{figure}

At every negative crossing of the diagram $D$, the outgoing semiarcs $c$ and $d$ are related to the incoming semiarcs $a$ and $b$ by two relations $c=a\upl{b}$ and $d=b\dol{a}$ (see the right part of Figure \ref{fig:slika5}). Figure \ref{fig:slika6} shows a homotopy between $\overline{j}(a\upl{b})$ and $\overline{j}(c)$ and another homotopy between $\overline{j}(b\dol{a})$ and $\overline{j}(d)$. This shows that $(\overline{j}\times \overline{j})(R)\subset \Delta _{\bi }$. 
\begin{figure}[h!]
\labellist
\normalsize \hair 2pt
\pinlabel $z_{0}$ at 220 650
\pinlabel $z_{1}$ at 265 10
\pinlabel $c_{0}\simeq a_{0}\overline{b}_{0}\overline{m}_{b_{0}(0)}b_{0}$ at 385 540 
\pinlabel $c_{1}\simeq a_{1}$ at 325 80
\pinlabel $b_{0}$ at 210 440 
\pinlabel $a_{0}$ at 150 540
\pinlabel $a_{1}$ at 190 80 
\pinlabel $z_{0}$ at 720 650
\pinlabel $z_{1}$ at 710 0
\pinlabel $b_{0}$ at 790 530 
\pinlabel $b_{1}$ at 780 90
\pinlabel $a_{1}$ at 740 150 
\pinlabel $d_{0}\simeq b_{0}$ at 600 530
\pinlabel $d_{1}\simeq b_{1}\overline{a}_{1}\overline{m}_{a_{1}(0)}a_{1}$ at 570 90 
\endlabellist
\begin{center}
\includegraphics[scale=0.5]{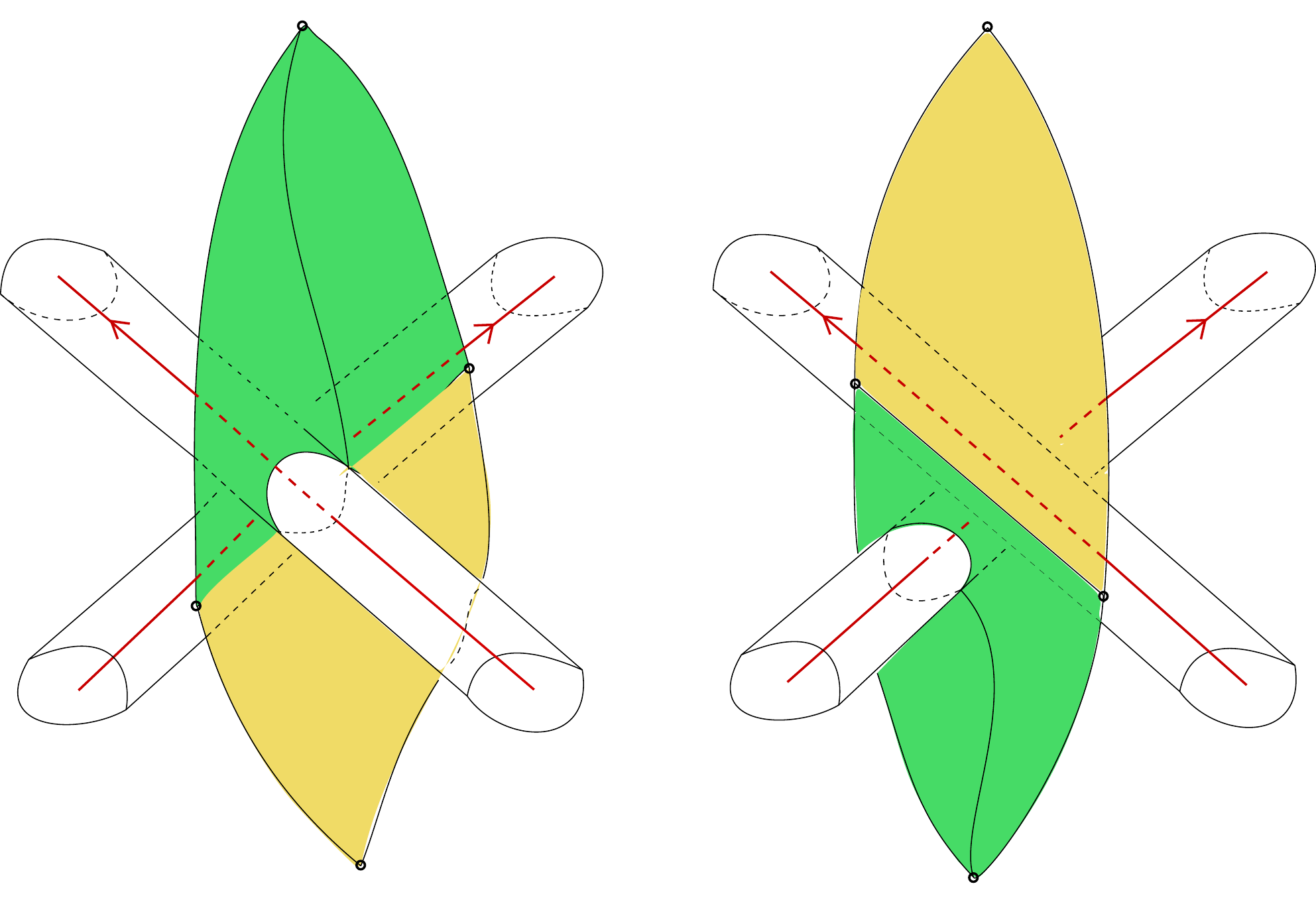}
\caption{An illustration of the crossing relations $[a_{0},a_{1}]^{\overline{[b_{0},b_{1}]}}=[c_{0},c_{1}]$ and $[b_{0},b_{1}]_{\overline{[a_{0},a_{1}]}}=[d_{0},d_{1}]$}
\label{fig:slika6}
\end{center}
\end{figure}

\item[(proof of 2.)] Suppose $Y$ is a biquandle and choose any map $f\colon A(D)\to Y$ such that $(\overline{f}\times \overline{f})(R)\subset \Delta _{Y}$. An element of $\bi $ is represented by a pair $(\gamma _{0},\gamma _{1})$, where $\gamma _{i}$ is a path in $E_{L}$ from a point in $\partial N_{L}$ to $z_{i}$ for $i=0,1$ and $\gamma _{0}(0)=\gamma _{1}(0)$. Project the paths $\gamma _{0}$, $\gamma _{1}$ in general position onto the plane of the diagram $D$. Suppose that the initial point $\gamma _{0}(0)=\gamma _{1}(0)$ lies on the parallel curve to the semiarc $a$ and suppose that $\gamma _{0}$ subsequently passes \underline{under} the semiarcs labelled by $b_{1},b_{2}\ldots ,b_{m}$, while $\gamma _{1}$  subsequently passes \underline{over} the semiarcs labelled by $c_{1},c_{2}\ldots ,c_{n}$. Define $$\widetilde{f}([\gamma _{0},\gamma _{1}]):=f(a)\up{f(b_{1})^{\epsilon _{1}}\ldots f(b_{m})^{\epsilon _{m}}}\dow{f(c_{1})^{\phi _{1}}\ldots f(c_{n})^{\phi _{n}}}\;,$$ where $\epsilon _{i}$ denotes the sign of the crossing between $\gamma _{0}$ and its overlying semiarc $b_{i}$, while $\phi _{i}$ denotes the sign of the crossing between $\gamma _{1}$ and its underlying semiarc $c_{i}$.     

It follows from the above definition of $\widetilde{f}$ that for any $a\in A(D)$, we have $(\widetilde{f}\circ j)(a)=\widetilde{f}([a_{0},a_{1}])=f(a)$, therefore $\widetilde{f}\circ j=f$. We need to show that $\widetilde{f}$ is a well defined map on $\bi $ and that it is a biquandle homomorphism. To show that $\widetilde{f}$ is well defined, we have to check that any representative of the equivalence class $[\gamma _{0},\gamma _{1}]$ gives the same value of $\widetilde{f}$. During a homotopy from $(\gamma _{0},\gamma _{1})$ to another representative $(\alpha _{0},\alpha _{1})$, the following critical stages may occur: \begin{enumerate}
\item The initial point $\gamma _{0}(0)=\gamma _{1}(0)$ moves to another semiarc. 

\begin{figure}[h!]
\labellist
\normalsize \hair 2pt
\pinlabel $a$ at 370 220
\pinlabel $d$ at 70 105 
\pinlabel $c$ at 290 10
\pinlabel $z_{1}$ at 400 -20
\pinlabel $z_{0}$ at 260 420
\pinlabel $\gamma _{0}$ at 310 240 
\pinlabel $\gamma _{1}$ at 380 80
\pinlabel $\alpha _{0}$ at 120 180
\pinlabel $\alpha _{1}$ at 180 20
\endlabellist
\begin{center}
\includegraphics[scale=0.4]{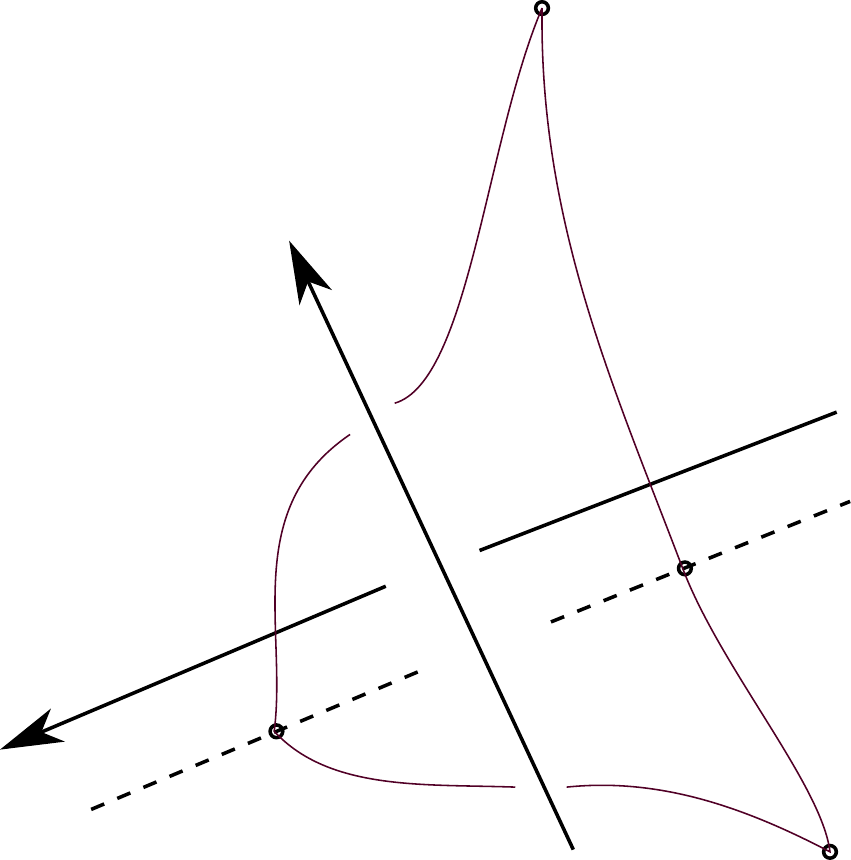}
\caption{The invariance of $\widetilde{f}$ - change of initial point}
\label{fig:slika7}
\end{center}
\end{figure}
First suppose that the initial point of $\gamma _{i}$ is at the semiarc $a$, while the initial point of $\alpha _{i}$ is at the semiarc $d$ where $a\up{c}=d$ (see Figure \ref{fig:slika7}). Since $(\overline{f}\times \overline{f})(R)\subset \Delta _{Y}$, we have $f(a)\up{f(c)}=f(d)$. Writing $\widetilde{f}([\gamma _{0},\gamma _{1}])=f(a)\up{w_{1}}\dow{w_{2}}$, we use Lemma \ref{lemma0} to obtain $\widetilde{f}([\alpha _{0},\alpha _{1}])=f(d)\upl{\overline{f}(c\dow{a})}\up{w_{1}}\dow{w_{2}}=f(a)\up{f(c)}\upl{f(c)\dow{f(a)}}\up{w_{1}}\dow{w_{2}}=f(a)\up{w_{1}}\dow{w_{2}}=\widetilde{f}([\gamma _{0},\gamma _{1}])$. 
\begin{figure}[h!]
\labellist
\normalsize \hair 2pt
\pinlabel $a$ at 370 190
\pinlabel $d$ at 70 65 
\pinlabel $c$ at 150 330
\pinlabel $z_{1}$ at 430 20
\pinlabel $z_{0}$ at 260 420
\pinlabel $\gamma _{0}$ at 320 330 
\pinlabel $\gamma _{1}$ at 390 100
\pinlabel $\alpha _{0}$ at 230 280
\pinlabel $\alpha _{1}$ at 280 70
\endlabellist
\begin{center}
\includegraphics[scale=0.4]{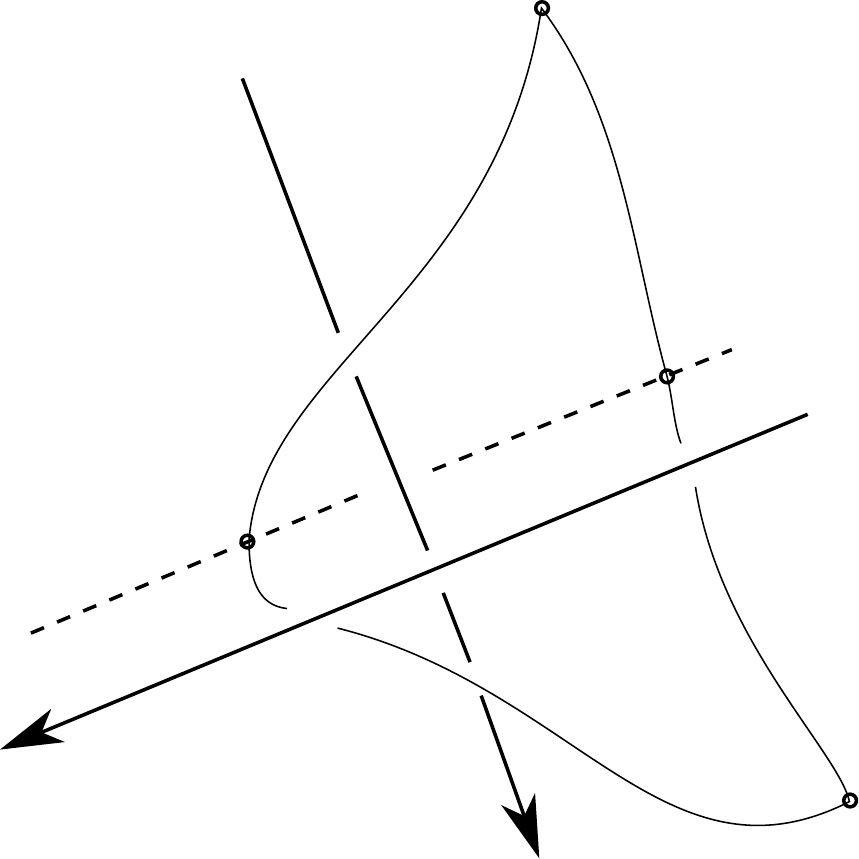}
\caption{The invariance of $\widetilde{f}$ - change of initial point}
\label{fig:slika8}
\end{center}
\end{figure}

Secondly, suppose that the initial point of $\gamma _{i}$ is at the semiarc $a$, while the initial point of $\alpha _{i}$ is at the semiarc $d$ where $a\dow{c}=d$ (see Figure \ref{fig:slika8}). Since $\overline{f}$ preserves the crossing relations, we have $f(a)\dow{f(c)}=f(d)$. Since $\overline{f}$ preserves the relations $R_{a,b,c}$, it follows by Proposition \ref{prop1} that any up-operation on $\overline{f}(A(D))$ commutes with any down-operation. Write $\widetilde{f}([\gamma _{0},\gamma _{1}])=f(a)\up{w_{1}}\dow{w_{2}}$ and it follows that $\widetilde{f}([\alpha _{0},\alpha _{1}])=f(d)\up{w_{1}}\dol{\overline{f}(c\up{a})}\dow{w_{2}}=f(a)\dow{f(c)}\dol{f(c)\up{f(a)}}\up{w_{1}}\dow{w_{2}}=f(a)\up{w_{1}}\dow{w_{2}}=\widetilde{f}([\gamma _{0},\gamma _{1}])$. 

For the two remaining cases, we prove the invariance in a similar way. 

\item The arc $\gamma _{0}$, overcrossed by the same semiarc $b$ twice, homotopes to an arc $\alpha _{0}$ that is not crossed by $b$ (or the arc $\gamma _{1}$, overcrossing the same semiarc $b$ twice, homotopes to an arc $\alpha _{1}$ that does not cross $b$). 
\begin{figure}[h!]
\labellist
\normalsize \hair 2pt
\pinlabel $a$ at 210 120
\pinlabel $b$ at 240 280
\pinlabel $z_{1}$ at 260 -15
\pinlabel $z_{0}$ at 230 435
\pinlabel $\gamma _{0}$ at 130 380 
\pinlabel $\gamma _{1}$ at 160 50
\pinlabel $\alpha _{0}$ at 300 340
\pinlabel $\alpha _{1}$ at 340 100
\endlabellist
\begin{center}
\includegraphics[scale=0.4]{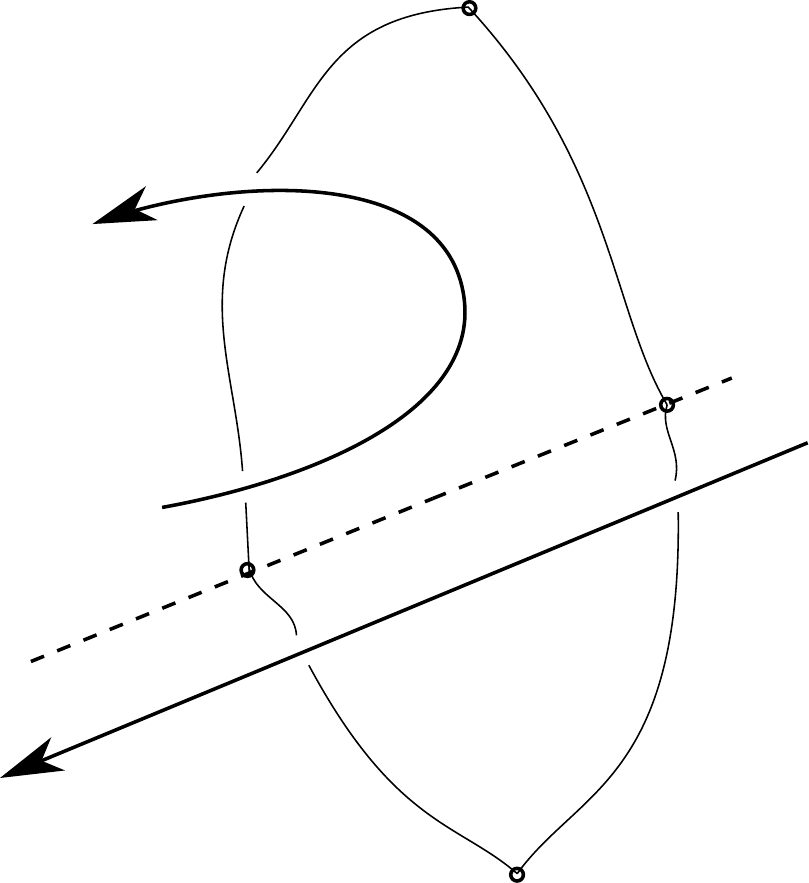}
\caption{The invariance of $\widetilde{f}$ under a homotopy - first case of (b)}
\label{fig:slika9}
\end{center}
\end{figure}
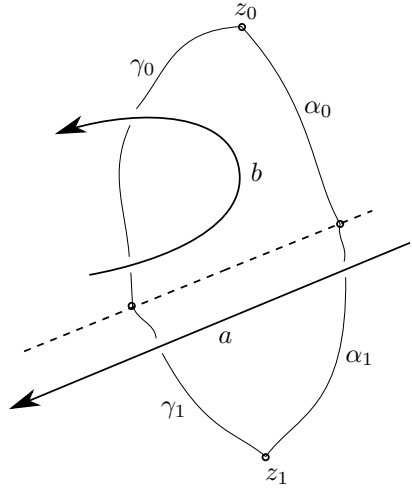

For the first case, see Figure \ref{fig:slika9}. We have $\widetilde{f}([\gamma _{0},\gamma _{1}])=f(a)\up{w_{1}}\up{f(b)}\upl{f(b)}\up{w_{2}}\dow{w_{3}}$ and $\widetilde{f}([\alpha _{0},\alpha _{1}])=f(a)\up{w_{1}}\up{w_{2}}\dow{w_{3}}$. Since $\overline{f}$ preserves the relations $R_{a,b,c}$, it follows by Proposition \ref{prop1} that $\widetilde{f}([\gamma _{0},\gamma _{1}])=\widetilde{f}([\alpha _{0},\alpha _{1}])$. 

\begin{figure}[h!]
\labellist
\normalsize \hair 2pt
\pinlabel $a$ at 210 310
\pinlabel $b$ at 230 130
\pinlabel $z_{1}$ at 235 -15
\pinlabel $z_{0}$ at 245 435
\pinlabel $\gamma _{0}$ at 170 380 
\pinlabel $\gamma _{1}$ at 110 130
\pinlabel $\alpha _{0}$ at 320 340
\pinlabel $\alpha _{1}$ at 315 100
\endlabellist
\begin{center}
\includegraphics[scale=0.4]{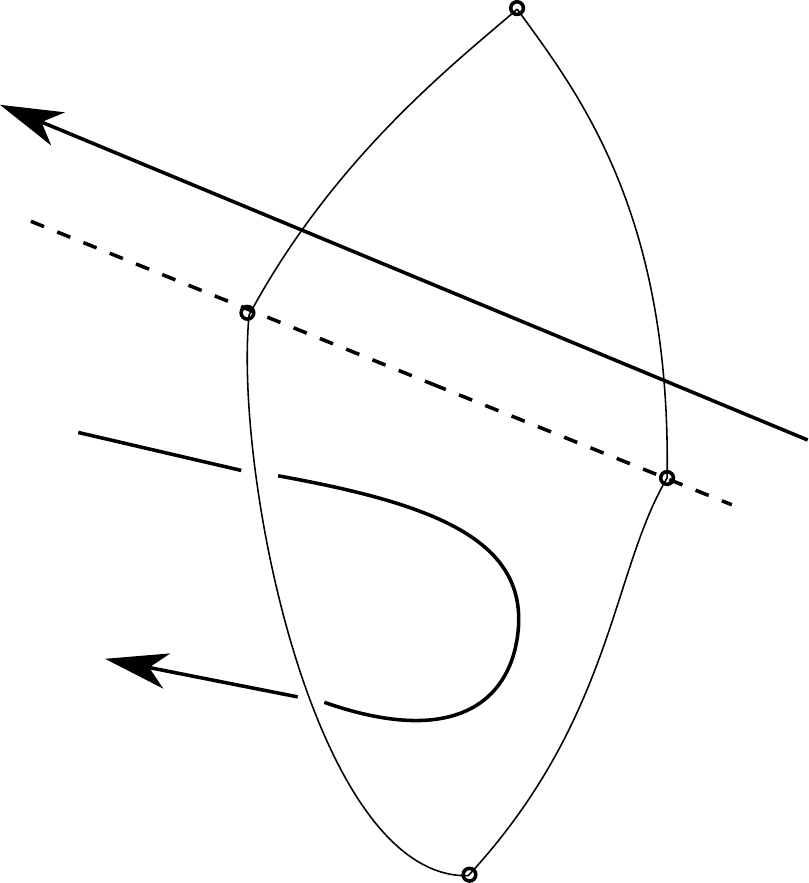}
\caption{The invariance of $\widetilde{f}$ under a homotopy - second case of (b)}
\label{fig:slika10}
\end{center}
\end{figure}
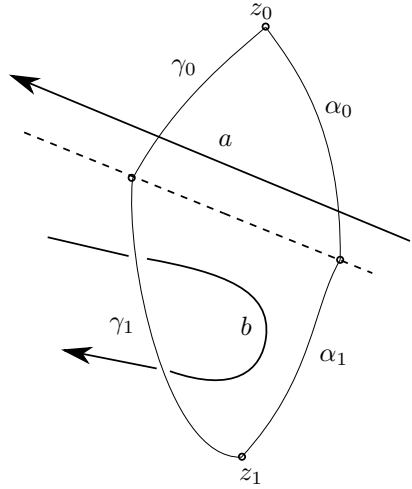
For the second case, see Figure \ref{fig:slika10}. We have $\widetilde{f}([\gamma _{0},\gamma _{1}])=f(a)\up{w_{1}}\dow{w_{2}}\dow{f(b)}\dol{f(b)}\dow{w_{3}}$ and $\widetilde{f}([\alpha _{0},\alpha _{1}])=f(a)\up{w_{1}}\dow{w_{2}}\dow{w_{3}}$. Since $\overline{f}$ preserves the relations $R_{a,b,c}$, it follows by Proposition \ref{prop1} that $\widetilde{f}([\gamma _{0},\gamma _{1}])=\widetilde{f}([\alpha _{0},\alpha _{1}])$.

\item $\gamma _{0}$ passes under a crossing between two semiarcs (or $\gamma _{1}$ passes over a crossing between two semiarcs). 
\begin{figure}[h!]
\labellist
\normalsize \hair 2pt
\pinlabel $a$ at 200 120
\pinlabel $b$ at 20 260
\pinlabel $c$ at 20 435
\pinlabel $c\dow{b}$ at 390 280 
\pinlabel $b\up{c}$ at 380 430
\pinlabel $z_{1}$ at 200 -15
\pinlabel $z_{0}$ at 200 585
\pinlabel $\gamma _{0}$ at 110 480 
\pinlabel $\gamma _{1}$ at 85 60
\pinlabel $\alpha _{0}$ at 280 480
\pinlabel $\alpha _{1}$ at 315 60
\endlabellist
\begin{center}
\includegraphics[scale=0.4]{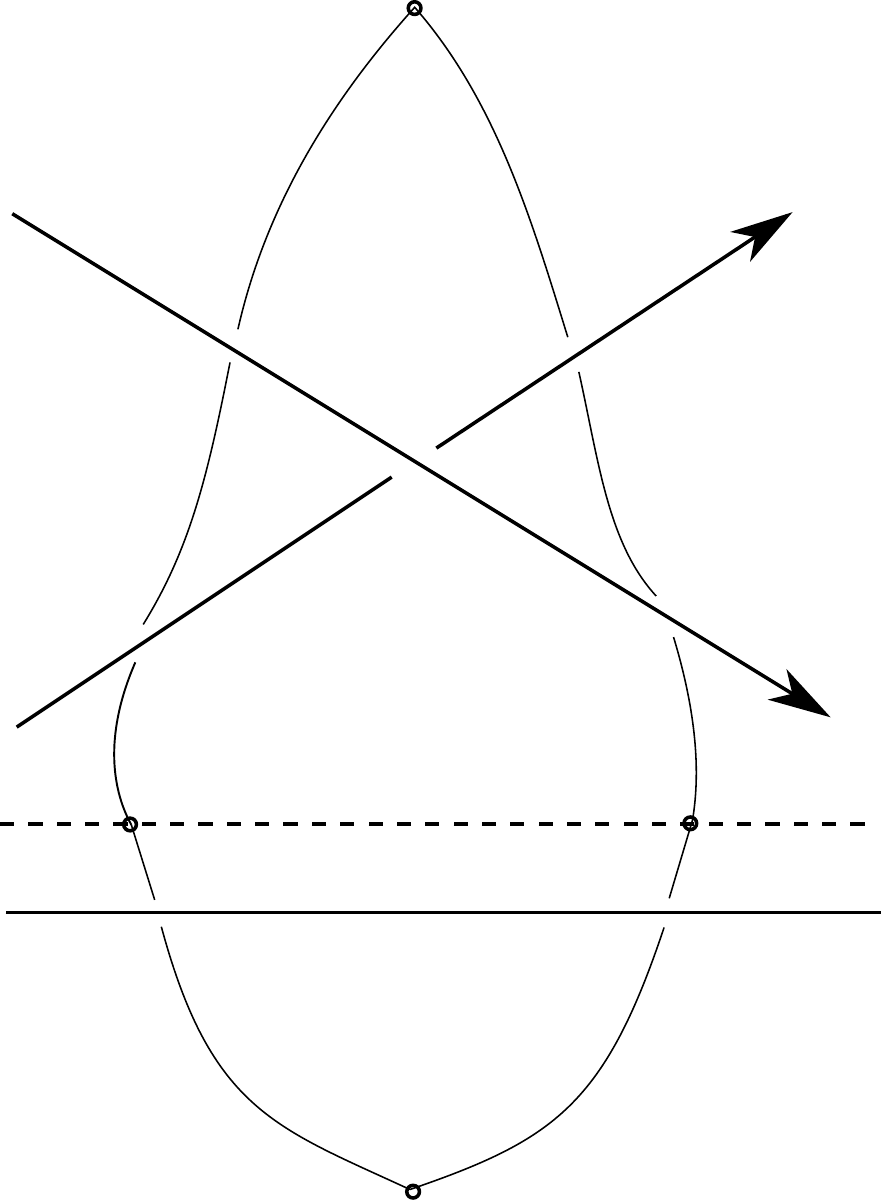}
\caption{The invariance of $\widetilde{f}$ under a homotopy - first case of (c)}
\label{fig:slika11}
\end{center}
\end{figure}

For the first case, see Figure \ref{fig:slika11}. We write $\widetilde{f}([\gamma _{0},\gamma _{1}])=f(a)\up{w_{1}}\up{f(b)}\up{f(c)}\up{w_{2}}\dow{w_{3}}$ and $\tilde{f}([\alpha _{0},\alpha _{1}])=f(a)\up{w_{1}}\up{f(c)\dow{f(b)}}\up{f(b)\up{f(c)}}\up{w_{2}}\dow{w_{3}}$. Since $\overline{f}$ preserves the relations $R_{a,b,c}$, we may use the first equality of the 3.Biquandle axiom to compute \begin{xalignat*}{1}
& f(a)\up{w_{1}}\up{f(c)\dow{f(b)}}\up{f(b)\up{f(c)}}=f(a)\up{w_{1}}\up{f(b)}\up{f(c)}
\end{xalignat*} and therefore $\widetilde{f}([\gamma _{0},\gamma _{1}])=\widetilde{f}([\alpha _{0},\alpha _{1}])$. The remaining cases are settled in a similar way. \\

To show that $\widetilde{f}$ is a biquandle homomorphism, choose two elements $[\alpha _{0},\alpha _{1}], [\beta _{0},\beta _{1}]\in \bi $. Let $\widetilde{f}[\alpha _{0},\alpha _{1}]=f(a)\up{w_{1}}\dow{w_{2}}$ and $\widetilde{f}[\beta _{0},\beta _{1}]=f(b)\up{z_{1}}\dow{z_{2}}$. Using Proposition \ref{prop1}, we calculate 
\begin{xalignat*}{1}
& \widetilde{f}\left ([\alpha _{0},\alpha _{1}]^{[\beta _{0},\beta _{1}]}\right )=\widetilde{f}[\alpha _{0}\overline{\beta }_{0}m_{\beta _{0}(0)}\beta_{0},\alpha _{1}]=f(a)\up{w_{1}}\upl{z_{1}}\up{f(b)}\up{z_{1}}\dow{w_{2}}=\\
& =f(a)\up{w_{1}}\dow{w_{2}}\upl{z_{1}}\up{f(b)}\up{z_{1}}=\widetilde{f}[\alpha _{0},\alpha _{1}]\up{f(b)\up{z_{1}}}=\widetilde{f}[\alpha _{0},\alpha _{1}]\up{\widetilde{f}[\beta _{0},\beta _{1}]}\\
& \widetilde{f}\left ([\alpha _{0},\alpha _{1}]_{[\beta _{0},\beta _{1}]}\right )=\widetilde{f}[\alpha _{0},\alpha _{1}\overline{\beta }_{1}m_{\beta _{1}(0)}\beta_{1}]=f(a)\up{w_{1}}\dow{w_{2}}\dol{z_{2}}\dow{f(b)}\dow{z_{2}}=\\
& =\widetilde{f}[\alpha _{0},\alpha _{1}]\dow{f(b)\dow{z_{2}}}=\widetilde{f}[\alpha _{0},\alpha _{1}]\dow{f(b)\dow{z_{2}}\up{z_{1}}}=\widetilde{f}[\alpha _{0},\alpha _{1}]\dow{\widetilde{f}[\beta _{0},\beta _{1}]}\;,
\end{xalignat*} thus $\widetilde{f}$ is indeed a biquandle homomorphism.

To prove uniqueness of $\widetilde{f}$, observe that by Corollary \ref{cor1}, any element of $\bi $ can be written as $[\gamma _{0},\gamma _{1}]=j(a)\up{\overline{j}(w_{1})}\dow{\overline{j}(w_{2})}$, where $a\in A(D)$ and $w_{1},w_{2}$ are elements of the free group, generated by $A(D)$. If $g\colon \bi \to Y$ is any biquandle homomorphism for which $g\circ j=f$, then we have $$g[\gamma _{0},\gamma _{1}]=f[a_{0},a_{1}]\up{\overline{f}(w_{1})}\dow{\overline{f}(w_{2})}=\widetilde{f}[\gamma _{0},\gamma _{1}]\;.$$
\end{enumerate}

\end{enumerate}
\end{proof}

\begin{corollary} For any link $L$, the topological biquandle $\bi $ is a quotient of its fundamental biquandle $BQ(L)$. 
\end{corollary}

\begin{corollary}The topological biquandle is a link invariant. 
\end{corollary}

\begin{example} \label{ex1} Consider the link $L=L6n1$ in the Thistlethwaite link table, whose diagram is depicted in Figure \ref{fig:slika12}. 
\begin{figure}[h!]
\labellist
\normalsize \hair 2pt
\pinlabel $a$ at 150 240
\pinlabel $b$ at 200 5
\pinlabel $c$ at 450 250 
\pinlabel $d$ at 300 300
\pinlabel $e$ at 390 120
\pinlabel $f$ at 620 210
\pinlabel $g$ at 250 330
\pinlabel $h$ at 260 200
\pinlabel $i$ at 240 110
\pinlabel $j$ at 360 200
\pinlabel $k$ at 365 330
\pinlabel $l$ at -5 210 
\endlabellist
\begin{center}
\includegraphics[scale=0.4]{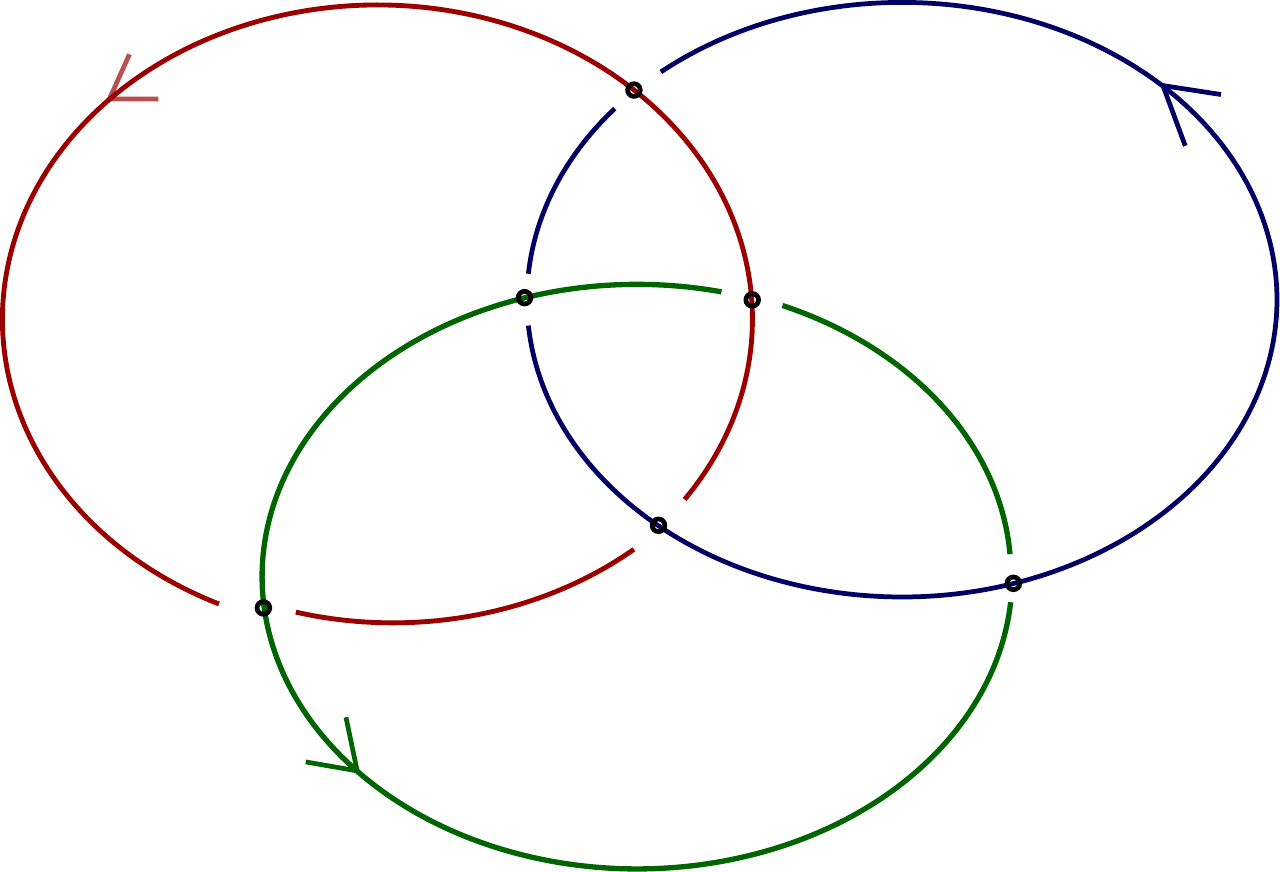}
\caption{A diagram of the link $L6n1$ from Example \ref{ex1}}
\label{fig:slika12}
\end{center}
\end{figure}
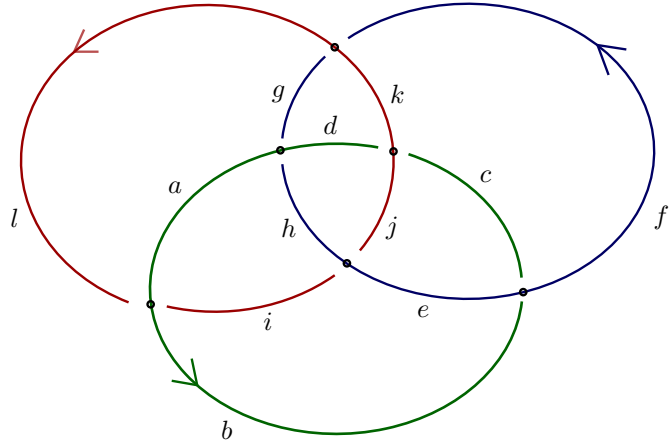
Denoting the semiarcs of the diagram as shown in the Figure \ref{fig:slika12}, the fundamental biquandle of $L$ is given by the presentation
\begin{xalignat*}{1}
& BQ(L)=\langle a,b,c,d,e,f,g,h,i,j,k,l\, |\, l\up a=i, a\dow l=b, f\up k=g, k\dow f=l, g\up d=h, \\
& d\dow g=a, c\up j=d, j\dow c=k, i\up h=j, h\dow i=e, b\up e=c, e\dow b=f\rangle \;,
\end{xalignat*} that reduces to
\begin{xalignat*}{1}
& BQ(L)=\langle b,f,l\, |\,b\up{f\dol{b}}\up{l\dol{f}\dol{b\up{f\dol{b}}}}\dow{f\up{l\dol{f}}}\dow{l}=b,\\
&  f\up{l\dol{f}}\up{b\dol{l}\dol{f\up{l\dol{f}}}}\dow{l\up{b\dol{l}}}\dow{b}=f,  l\up{b\dol{l}}\up{f\dol{b}\dol{l\up{b\dol{l}}}}\dow{b\up{f\dol{b}}}\dow{f}=l \rangle \;.
\end{xalignat*}
The topological biquandle $\bi $ is given by the presentation 
\begin{xalignat*}{1}
& \bi =\langle a,b,c,d,e,f,g,h,i,j,k,l\, |\, l\up a=i, a\dow l=b, f\up k=g, k\dow f=l, g\up d=h, \\
& d\dow g=a, c\up j=d, j\dow c=k, i\up h=j, h\dow i=e, b\up e=c, e\dow b=f, R\rangle \;,
\end{xalignat*}
where $R$ denotes all relations $R_{x,y,z}$ for $x,y,z\in \{a,b,c,d,e,f,g,h,i,j,k,l\}$. These relations include: $x\up{y\dol{x}}=x\up{y}$, $x\up{y\dol{z}\dol{w\up{z}}}=x\up{y}$ and $x\dow{y\up{z}}=x\dow{y}$ for every $x,y,z,w\in \{b,f,l\}$. 
 Since none of these new relations is implied from the relations in the presentation of $BQ(L)$, it follows that the topological biquandle $\bi $ is a quotient of the fundamental biquandle $BQ(L)$. The presentation of the topological biquandle thus reduces to
\begin{xalignat*}{1}
& \bi =\langle b,f,l\, |\, b\up f\up l\dow f\dow l=b, f\up l\up b\dow l\dow b=f,  l\up b\up f\dow b\dow f=l, R\rangle \;.
\end{xalignat*}  
\end{example}

\begin{remark}A presentation of the topological biquandle $\bi $ is obtained from a presentation of the fundamental biquandle $BQ(L)$ by adding relations 
\begin{xalignat*}{1}
R_{a,b,c}=\left \{ a\up{b\dow{c}}=a\up{b},\, a\upl{b\dow{c}}=a\upl{b},\, a\dow{b\up{c}}=a\dow{b},\, a\dol{b\up{c}}=a\dol{b}\right \}
\end{xalignat*}
 for every ordered triple of generators $(a,b,c)$. Seeing $\bi $ as a subbiquandle of the fundamental biquandle $BQ(L)$, we may talk about the corresponding ''sections''. For any $a\in BQ(L)$, the section $\bi a$ is given as $\bi a=\{x\up a,x\dow a|\, x\in \bi \}$. The quotient set $BQ(L)/\bi $ is generated by $$BQ(L)/\bi =\left \langle \bi a\up b,\, \bi a\dow b,\, \bi a\upl b,\, \bi a\dol b\, |\, a,b\in BQ(L)\right \rangle \;.$$ Denoting by $n$ the number of generators of $BQ(L)$, the quotient set $BQ(L)/\bi $ has $4n^{2}$ generators, which indicates the ''index'' of the topological biquandle inside the fundamental biquandle. In Example \ref{ex1}, the quotient $BQ(L)/\bi $ has $36$ generators. 
\end{remark}

One might question the need for the topological biquandle, when the fundamental quandle is already a complete invariant of knots up to inversion. In a more sophisticated study of links (e.g. virtual links), however, we sometimes need to combine two or more different link invariants to yield a stronger invariant. Some examples of this are the quantum enhancements using biquandles, see \cite{NEL1, NEL2, NEL3, MAN1}. In the study of virtual links, Manturov introduced the concept of parity \cite{MAN2}, that induces a function on the set of crossings of any virtual link diagram. Parity allows constructions of new link invariants and also improvement of the existing invariants (e.g. Kauffman bracket). As was shown in \cite[Example 2.3]{MAN1}, a parity of knots may be induced by a certain coloring of the fundamental biquandle of the knot. It might be possible to define other parities of virtual knots using the fundamental or topological biquandle. 

The topological biquandle may just as well be defined for links in other 3-manifolds, virtual links, or higher-dimensional links, and it might lead to interesting new invariants.

\end{section}

\section*{Acknowledgements}The author was supported by the Slovenian Research Agency grant N1-0083.

\end{document}